\documentclass[english]{amsart}
\usepackage[T1]{fontenc}
\usepackage[latin9]{inputenc}
\usepackage{amsmath}
\usepackage{amssymb}
\usepackage{amsmath}
\usepackage{amstext}
\usepackage{amsfonts}
\usepackage{amscd}
\usepackage{amsbsy}
\usepackage{amsthm}
\usepackage{ifthen}
\usepackage{enumerate}
\usepackage{latexsym}
\usepackage{babel}

\makeatletter

\makeatletter
\def\input@path{{/home/filiuspelei/Hope//}}
\makeatother

\newtheorem*{utheorem}{Theorem}
\newtheorem{prop}{Proposition}[section]
\newtheorem{lemma}{Lemma}[section]
\newtheorem{theorem}{Theorem}[section]

\newtheorem{definition}{Definition}[section]
\newtheorem{example}{Example}
\newtheorem{corollary}{Corollary}[section]

\newtheorem{question}{Question}

\DeclareMathOperator{\hm}{Hom}
\DeclareMathOperator{\ext}{Ext}

\DeclareMathOperator{\Char}{char}

\DeclareMathOperator{\edim}{edim}

\newcommand{\NN}{\mathbb{N}}

\newcommand{\F}{\mathfrak{F}}

\newcommand{\al}{\alpha}
\newcommand{\be}{\beta}
\newcommand{\ga}{\gamma}

\newcommand{\ep}{\varepsilon}
\newcommand{\ze}{\zeta}

\newcommand{\ka}{\kappa}
\newcommand{\ld}{\lambda}

\newcommand{\ph}{\varphi}

\newcommand{\sbe}{\subseteq}

\newcommand{\tns}{\otimes}
\newcommand{\x}{\times}

\newcommand{\del}{\backslash}

\begin{document}

\title{Semidualizing Modules and Rings of Invariants}

\author{Billy Sanders}

\begin{abstract}
We show there exist no nontrivial semidualizing modules for nonmodular rings of invariants of order $p^n$ with $p$ a prime.
\end{abstract}


\maketitle

\section{Introduction}

This paper is concerned with the existence of nontrivial semidualizing modules.  Recall 
\begin{definition}

A finitely generated $S$-module $C$ is semidualizing if the map $S\to\hm_S(C,C)$ given by $s\mapsto(x\mapsto sx)$ is an isomorphism and $\ext_S^{i>0}(C,C)=0$.

\end{definition}
This is equivalent to saying $S$ is totally $C$ reflexive.  Examples always include $S$ and the dualizing module, if it exists, thus we call these semidualizing modules.  Semidualizing modules were first discovered by Foxby in \cite{Foxby72}.  They were later rediscovered by various other authors including Vasconcelos, who called them spherical modules, and Golod, who referred to them as suitable modules.  In \cite{Vasconcelos74}, Vasconcelos asks if there exists only a finite number of nonisomorphic semidualizing modules.  This question is answered in the affirmative in \cite{Christensen08} for  equicharacteristic  Cohen-Macaulay algebras, and in \cite{Nasseh12} for the semilocal case.  Since their discovery, semidualizing modules have been the focus of much research.  See for example \cite{AryaTakahashi09},\cite{Gerko05},\cite{Vasconcelos74},\cite{Nasseh12},\cite{Jorgensen12},\cite{Sather-Wagstaff09}, and \cite{Sather-Wagstaff09b}. 

It is natural to ask  which rings have only trivial semidualizing modules.  In \cite{Jorgensen12}, Jorgensen, Leuschke and Sather-Wagstaff  give a very nice characterization of rings with  a dualizing module and  only trivial semidualizing modules.  However, this characterization is somewhat abstract and it is difficult to tell whether the conditions hold for a particular ring.  Also in \cite{Sather-Wagstaff09}, Sather-Wagstaff proves results relating the existence of nontrivial semidualizing modules to Bass numbers. 
In this paper, we pose the following question:

\begin{question}

If a ring $S$ has a nice (e.g. rational) singularity, then does $S$ have  only trivial semidualizing modules?

\end{question}

The evidence suggests the answer is yes.   In \cite{Dao11}, Celikbas and Dao  show that only trivial semidualizing modules exist over Veronese subrings, which have a quotient singularity and hence a rational singularity.  Furthermore, Sather-Wagstaff shows in \cite{Sather-Wagstaff07} that only trivial semidualizing modules exist for determinantal rings, which also have a rational singularity.    It is proven in \cite{Sather-Wagstaff09b}[Example 4.2.14] that all Cohen-Macaulay rings with minimal multiplicity have no nontrivial semidualizing modules.  Since rational singularity and dimension 2 imply minimal multiplicity, 
all rings with rational singularity and dimension 2 have no nontrivial semidualizing modules.  
%
%
%
%
%
%
%
The following example shows that there are dimension 3 rings with rational singularity that do not have minimal multiplicity.

\begin{example}

Let 
$$S=k[[x,y,z]]^{(3)}=k[[x^3,y^3,z^3,x^2y,x^2z,y^2x,y^2z,z^2x,z^2y,xyz]]$$ which is the third Veronese subring in three variables.  For the multiplicity of $S$ to be minimal, it must equal $\edim S-\dim S+1=10-3+1=8$.  However, setting $\bar{S}=S/(x^3,y^3,z^3)S$, $e(S)=e(\bar{S})=\lambda(\bar{S})$ where $\ld$ is length.  Since 
$$\bar{S}=k\oplus kx^2y\oplus kx^2z\oplus ky^2x\oplus ky^2z\oplus kz^2x\oplus kz^2y\oplus  kxyz\oplus  kx^2y^2z^2$$
we thus have $e(S)=9$.

\end{example}

In this paper, we  add to the evidence that suggests that the answer to Question 1 is yes by investigating the case where $S$ is a ring of invariants, a large class of rings with rational singularity.   The following theorem is the main result of this paper.  

\begin{utheorem}

If $S$ is a power series ring over a field $k$ in finitely many variables and $G$ is a cyclic group of order $p^l$ acting on $S$ with $\Char k\ne p$, then $S^G$ has only trivial semidualizing modules.

\end{utheorem}

Our approach to the proof of this result, relying on Lemma \ref{lemma}, is different than those of the results in \cite{Dao11} and \cite{Sather-Wagstaff07}.  In each of those papers, the key technique involves counting the number of generators, whereas we use Lemma \ref{lemma}.  See Section 2 for a further explanation.
 
Section 2 gives preliminary results concerning rings of invariants and semidualizing modules and also gives a sketch of the proof.  Section 3 proves a key technical theorem about when a ring has only trivial semidualizing modules, and then Section 4 uses this result to prove our main theorem.

All rings considered in this paper will be Noetherian and commutative.

\section{Preliminaries}  

In this section, let $S$ be a Noetherian ring.  The proof relies upon the following lemma from \cite{Jorgensen12}.

\begin{lemma}\label{lemma}

If $C$ is a semidualizing $S$-module and $D$ is a dualizing module for $S$, then the homomorphism $\eta:C\tns\hm_S(C,D)\to D$ given by $x\tns \ph\mapsto \ph(x)$ is an isomorphism.

\end{lemma}

The map $\eta$ being an isomorphism is a strong condition since $D$ is torsionless and since tensor products often have torsion elements.  We will exploit this map using the following  lemma  from \cite{Sather-Wagstaff07}[Fact 2.4] and \cite{Gerko04}[Theorem 3.1].

\begin{lemma}

If $C$ is a semidualizing $S$-module and $S$ is a normal domain, then $C$ is reflexive and hence an element of the class group.

\end{lemma}

Therefore, when $S$ is normal, $\hm(C,D)$ is the element of the class group associated with $C^{-1}\circ D$, and  all three modules involved in Lemma \ref{lemma} are elements of the class group. In Theorem \ref{theorem1},  with strong assumptions on $S$, we show that $A\tns B$ has torsion for any elements $A$ and $B$ in the class group of $S$  which are not isomorphic to S.  The construction of  a torsion element is easy, however, it requires considerable work to show that this element is not zero in the tensor product.  With this setup, because of Lemma \ref{lemma} and since $D$ does not have torsion, nontrivial semidualizing modules cannot exist. The proof also requires the following lemma which is easily  proven  in \cite{Sather-Wagstaff09b}[Proposition 2.2.1].

\begin{lemma}\label{completion}

If $R\to S$ is a faithfully flat extension, then $C$ is a semidualizing $R$-module if and only if $C\tns S$ is a semidualizing $S$-module.

\end{lemma}

For the remainder of this paper, let  $R$ be a polynomial ring in finitely many variables over an algebraically closed field $k$, and let   $G$ be a finite group acting linearly on $R$.  We shall assume that the characteristic of $k$ does not divide the order of the group.   To prove the main result, Section 4 shows that when $|G|=p^l$ for some prime, $R^G$ satisfies the assumptions of Theorem \ref{theorem1}.  In order to do this, we need the following definition and lemma.

\begin{definition}

Given a character $\chi:G\to k^\x$, we denote by $R_\chi$ the set of relative invariants, namely, the polynomials $f\in R$ such that $gf=\chi(g) f$.  

\end{definition}

Note that $R_\chi$ is an $R^G$-module.  The following lemma is from \cite{Benson93}[Theorem 3.9.2].

\begin{lemma}

The ring $R^G$ is a normal domain whose  class group is the subgroup $H\sbe \hm(G,k^\x)$ which consists of the characters that contain all the pseudoreflections in their kernel.  Furthermore, for any $\chi\in H$, the relative invariants $R_{\chi^{-1}}$ form the reflexive module corresponding to the element $\chi$.  

\end{lemma}

\section{Class Groups}

In this section, let $S$ be a Noetherian ring.  We say that an element $\mu$ in an $S$-module $M$ is indivisible if there exists no nonunit $a\in S$ and $\nu\in M$ such that $\mu=a\nu$.  

\begin{lemma}

Suppose $S$ is a $k$-algebra, with $k$ a field and $M$ and $N$ are $S$-modules.  Furthermore, suppose $f\in M$ and $g\in N$ are indivisible, and $\ga\in M$ and $\rho\in N$  are not unit multiples of $f$ and $g$ respectively.
If there exists $k$-bases $E,F,X$ of $M,N, S$ respectively with $f,\ga\in E$ and $g,\rho\in F$ such that for every $\xi\in X$ and $\ep\in E$ and $\eta\in F$, $\xi\ep$ is a $k$-linear multiple of an element in $E$ and $\xi\eta$ is a $k$-linear multiple of an element of $F$, then $f\tns g-\ga\tns\rho$ is not zero in $M\tns_{S} N$.

\end{lemma}

\begin{proof}

Suppose that such bases $E,F,X$ exist.  Let $\F$ denote the free abelian group functor.  Recall that for any modules $U$ and $V$ over a ring $R$, we construct $U\tns_R V$ by quotienting $\F(U\cup V)$ by the submodule, which we will call $K_{U,V}(R)$, generated by the relations of the form
$$(v_1,u_1+u_2)-(v_1,u_1)-(v_1,u_2)$$
$$(v_1+v_2,u_1)-(v_1,u_1)-(v_2,u_1)$$
$$(\ld v_1,u_1)-(v_1,\ld u_1)$$ 
with $v_i\in U$, $u_i\in V$ and $\ld\in R$.  Hence, $M\tns_S N\cong \F(M\cup N)/K_{M,N}(S)$ and $M\tns_k N\cong \F(M\cup N)/K_{M,N} (k)$.  Notice that, since $k\sbe S$,  $K_{M,N}(k)\sbe K_{M,N}(S)$.  So $M\tns_S N$ is a quotient of $M\tns_k N$.  Specifically, we have the following isomorphism 
$$\frac{M\tns_k N}{K_{M,N}(S)/K_{M,N}(k)}\cong \frac{\F(M\cup N)/K_{M,N}(k)}{K_{M,N}(S)/K_{M,N}(k)}\cong \frac{\F(M\cup N)}{K_{M,N}(S)}\cong M\tns_S N$$

We claim that every element of $K_{M,N}(S)/K_{M,N}(k)\sbe M\tns_k N $ is of the form 
$$\sum_{s=1}^r \ld_s(\mu_s\tau_s\tns\nu_s)-\ld_s(\mu_s\tns\tau_s\nu_s)$$
with $\ld_i\in k$, and $\mu_i\in E$, $\nu_i\in F$, $\tau_i\in X\del k$  and $\ld_i\in k$.  Take $z\in K(S)/K(k)$.  Since the generators of $K(S)$ of the form  $(v_1,u_1+u_2)-(v_1,u_1)-(v_1,u_2)$ and $(v_1+v_2,u_1)-(v_1,u_1)-(v_2,u_1)$ are in $K(k)$, we may write $$z=\sum_i (m_it_i\tns n_i-m_i\tns t_in_i)$$ with $m_i\in M$, $n_i\in N$, and $t_i\in S$.  However, since $E,F,X$ are bases of $M,N,X$ respectively, we may also write 
$$m_i=\sum_j \al_{i,j}\mu_{i,j}\quad\quad\quad n_i=\sum_l \be_{i,l}\nu_{i,l}\quad\quad\quad t_i=\sum_k \ka_{i,k}\tau_{i,k}$$  
with each $\ld_s\in k$, and $\mu_s\in E$, $\nu_s\in F$, $\tau_s\in X\del k$  and $\ld_s\in k$.  So we have 
\begin{align*}
z
	&=\sum_i (m_it_i\tns n_i-m_i\tns t_in_i)\\
	&=\sum_i\left(\left(\sum_j \al_{i,j}\mu_{i,j}\right)\left(\sum_k \ka_{i,k}\tau_{i,k}\right)\tns\left(\sum_l \be_{i,l}\nu_{i,l}\right)-\left(\sum_j \al_{i,j}\mu_{i,j}\right)\tns\left(\sum_k \ka_{i,k}\tau_{i,k}\right)\left(\sum_l \be_{i,l}\nu_{i,l}\right)\right)\\
	&=\sum_{i,j,k,l}\left( \al_{i,j}\mu_{i,j}\ka_{i,k}\tau_{i,k}\tns\be_{i,l}\nu_{i,l}-\al_{i,j}\mu_{i,j}\tns\ka_{i,k}\tau_{i,k}\be_{i,l}\nu_{i,l}\right)\\
	&=\sum_{i,j,k,l}\al_{i,j}\be_{i,l}\ka_{i,k}\left( \mu_{i,j}\tau_{i,k}\tns\nu_{i,l}-\mu_{i,j}\tns\tau_{i,k}\nu_{i,l}\right)\\
	&=\sum_{i,j,k,l}\al_{i,j}\be_{i,l}\ka_{i,k}( \mu_{i,j}\tau_{i,k}\tns\nu_{i,l})-\al_{i,j}\be_{i,l}\ka_{i,k}(\mu_{i,j}\tns\tau_{i,k}\nu_{i,l})\\
\end{align*}
Lastly, if $\tau_{i,k}$ is in $k$, then $ \mu_{i,j}\tau_{i,k}\tns\nu_{i,l}-\mu_{i,j}\tns\tau_{i,k}\nu_{i,l}$ is already zero in $M\tns_k N$.  Therefore, setting $\ld_{i,j,k,l}=\al_{i,j}\be_{i,l}\ka_{i,k}\in k$, the claim is shown.  

Now suppose $f\tns g-\ga\tns\rho$ is zero in $M\tns_S N$.  Then in $M\tns_k N$, we may write 
$$f\tns g-\ga\tns\rho=\sum_{s=1}^r \ld_s(\mu_s\tau_s\tns\nu_s)-\ld_s(\mu_s\tns\tau_s\nu_s)$$ 
with $\ld_s\in k$, and $\mu_s\in E$, $\nu_s\in F$, $\tau_s\in X\del k$  and $\ld_s\in k$. 
Now  $Z=\{a\tns b\mid a\in E, b\in F\}$ is a $k$-basis of $M\tns_k N$.   Since $f,\ga\in E$ and $g,\rho\in F$, $f\tns g$ and $\ga\tns\rho$ are in $Z$.   By assumption, each $\mu_s\tau_s\tns\nu_s$ and $\mu_s\tns\tau_s\nu_s$ is a linear multiple of an element in $Z$.  Thus, $f\tns g$ must be a linear multiple of either $\mu_s\tau_s\tns\nu_s$ or $\mu_s\tns\tau_s\nu_s$ for some $s$.  But, since $f$ and $g$ are indivisible and for all $s$, neither $\mu_s\tau_s$ nor $\tau_s\nu_s$ is indivisible, this is a contradiction.  Therefore, $ f\tns g-\ga\tns\rho$ cannot be  zero in $M\tns_S N$.

\end{proof}

Take a ring $S$ with class group $L$ with operation $\circ$.  Let $T=\bigoplus_{A\in L} A$.  We can give this $S$-module an $L$-graded $S$-algebra structure.  For any $A,B\in L$, recall that $A\circ B=\hm(\hm(A\tns_S B,S),S)\in L$.  We will define the multiplication on the homogenous elements of $T$ with the natural map $\ph_{A,B}: A\tns_S B\to \hm(\hm(A\tns_S B,S),S)$ by setting $ab=\ph_{A,B}(a\tns b)$, for any $a\in A$ and $b\in B$.  We can extend this multiplication linearly to the nonhomogenous elements of $T$.  Since $S$ is contained in $T$, this algebra is unital,  and, because $\hm(\hm(A\tns_S B,S),S)\cong\hm(\hm(B\tns_S A,S),S)$, it is commutative as well.  This construction is similar to an algebra considered in \cite{TomariWatanabe92}.

\begin{theorem}\label{theorem1}

Let $S$ be  a Noetherian $k$-algebra, with $k$ a field.  Suppose $L$ is finite and cyclic with generator $\Lambda$.  Also suppose that the $L$-grading on $T$ can be refined to a grading $\Gamma$ such that every $\Gamma$-homogenous component is one dimensional.  If there exists a $\Gamma$-homogenous  element $x\in \Lambda\sbe T$ such that $x^n\in \Lambda^n\sbe T$ is indivisible (as an element of an $S$-module) for all $n\in \NN$ strictly less than $|\Lambda|$, then for any  $A,B\in L$ where neither $A$ nor $B$ is isomorphic to $S$,  the module $A\tns_S B$ has torsion. 

\end{theorem}

\begin{proof}

Since $\Lambda$ generates $L$, there exists $a$ and $b$ such that $\Lambda^a=A$ and $\Lambda^b=B$.  Then there exists $a,b\in\NN$ such that $x^a\in A$ and $x^b\in B$.  Since neither $A$ nor $B$ is isomorphic to $S$, $a$ and $b$ are both strictly less than $|L|$ and so $x^a$ and $x^b$ are indivisible.     We may assume without loss of generality that $a\ge b$.  

Let $Q$ a minimal homogenous generating set of $B$ which contains $x^b$.  We may assume every element in $Q$ is indivisible, since, by the Noetherian condition, we can replace any divisible element by an indivisible one.  Since $B$ is not isomorphic to $S$ and is torsionless, we know that $Q$ has another element $y$ besides $x^b$.  Besides being indivisible and homogeneous, $y$ is also not a unit multiple of $x^b$. 

Set $z=x^a\tns y -yx^{a-b}\tns x^b$.  We show that $z$ is a torsion element.  Since $x^{a-b}$ is in $\Lambda^{a-b}$ and $y$ is in $B=\Lambda^b$, $yx^{a-b}$ is $\Lambda^a$ which is $A$.  Thus $z$ is in $A\tns_S B$.  Furthermore,  for any $f\in (A\circ B)^{-1}$ we have $x^ayf,x^{a+b} f\in S$.  Thus we have, 
$$(x^ayf)z=x^{2a}yf\tns y-yx^{a-b}\tns x^{a+b}yf=x^{2a}yf\tns y-x^{2a}yf\tns y=0$$
Thus to show that $z$ is a torsion element, 
 it suffices to show that $z$ is not zero in $A\tns_S B$.  

Note that, by construction, $x^a$ and $y$ are indivisible, and since $y$ and $x^b$ are not unit multiples of each other, neither are $x^a$ and $yx^{a-b}$.  Also $yx^{a-b}$ is homogenous since $x^{a-b}$ is.  We can choose $\Gamma$-homogenous bases $E$ and $F$ of $A$ and $B$ respectively such that $x^a,y\in E$ and $x^a,yx^{a-b}\in F$.  Similarly we can choose a $\Gamma$-homogenous basis $X$ of $S$.  Since every $\Gamma$-homogenous component of $T$ is one dimensional, for every $\xi\in X$ and $\ep\in E$ and $\eta\in F$, $\xi\ep$ is a linear multiple of an element in $E$ and $\xi\eta$ is a linear multiple of an element of $F$.  Thus $z$ meets the hypotheses of the previous proposition.  Therefore, $z$ is not zero in $A\tns_S B$.

\end{proof}

\begin{corollary}\label{main}

Assume the set up of the last Theorem and that $S$ has a dualizing module.  Then $S$ has no nontrivial semidualizing modules.

\end{corollary}

\begin{proof}

Let $C$ be a semidualizing module for $S$.  Then $C\tns \hm (C,D)\cong D$  
where $D$ is a dualizing module.  However, $\hm(C,D)\cong C^{-1}\circ D$ is also an element of the class group. Thus by the previous theorem, since $D$ is torsionless, either $C$ or $\hm(C,D)$ is isomorphic to $S$.  Therefore, $C$ is isomorphic to $S$ or $D$.
 
\end{proof}

\section{Semidualizing Modules of Rings of Invariants}

Let $R$ be the polynomial ring in $d$ variables over $k$.  We can apply the previous results to the semidualizing modules over rings of invariants for a certain cyclic group, but first we need a lemma.  

\begin{lemma}\label{lemma2}

Assume $k$ is an algebraically closed field.  If $G$ is a finite cyclic group  acting linearly on $R$ generated by $g$ whose order is not divisible by the characteristic of $k$, then there exist algebraically independent $x_1,\dots,x_d\in R$ such that $R=k[x_1,\dots,x_d]$ and $gx_i=\zeta^{\eta_i}x_i$ with $\zeta\in k$ a primitive $|G|$th root of unity.  

\end{lemma}

\begin{proof}

By putting $g$ in the Jordan canonical form, it is an easy exercise to see that $g$ is diagonalizable since $|G|$ and $\Char k$ are coprime.   Thus, we may choose an eigenbasis, $x_1,\dots,x_d$, of $R_1$.  So, $gx_i=\xi_ix_i$ with $\xi_i\in k$.  Since $g^{|G|}$ should act as the identity, each $\xi_i$ must be a $|G|$th root of unity, and so we may write $\xi=\zeta^{\eta_i}$ where $\zeta$ is some fixed primitive $|G|$th root of unity.  Also, $R$ is isomorphic to the symmetric algebra of $R_1$ which is a polynomial ring in the variables $x_1,\dots,x_d$.  Hence,  $x_1,\dots,x_d$ are algebraically independent.  
 
\end{proof}

To apply the results of Section 3, we will observe that in this case
$$T=\bigoplus_{\chi\in L} R_{\chi^{-1}}\sbe R$$  
where $L$ is the class group of $R^G$.  The desired grading $\Gamma$ of $T$ will be the monomial grading with respect to the variables $x_1,\dots,x_d$ defined in the previous lemma.  Before we proceed however, we need to show that this grading is a refinement of $L$, to which end, the following lemma suffices.

\begin{lemma}

If $G$ consists of diagonal matrices, then for any character $\chi:G\to k^\times$, the set of all monomials in $R_\chi$ is a $k$-basis.  

\end{lemma}

\begin{proof}

Let $X$ be the set of all monomials of $R_\chi$.  Since any distinct monomials are linearly independent, $X$ is linearly independent.  Take any $g\in G$.  
Then for each $i$, $gx_i=\ld_ix_i$ with $\ld_i\in k$.  So for any $\underline{x}^{\underline{\al}}=x_1^{a_1}\cdots x_d^{a_d}$ in $R$, we have 
$$g\underline{x}^{\underline{\al}}=gx_1^{a_1}\cdots x_n^{a_d}=(\ld_1x_1)^{a_1}\cdots (\ld_dx_d)^{a_d}=\underline{\ld} x_1^{a_1}\cdots x_d^{a_d}=\underline{\ld}\underline{x}^{\underline{\al}}$$ 
with $\underline{\ld}=\ld_1^{a_1}\cdots \ld_d^{a_d}$.  Take any $f\in R_\chi$.  We may write $f=\kappa_1\underline{x}^{\underline{\al}_1}+\cdots+\kappa_m\underline{x}^{\underline{\al}_m}$.   On the one hand, we know that 
$$gf=g(\kappa_1\underline{x}^{\underline{\al}_1}+\cdots+\kappa_m\underline{x}^{\underline{\al}_m})
=g\kappa_1\underline{x}^{\underline{\al}_1}+\cdots+g\kappa_m\underline{x}^{\underline{\al}_m}
=\kappa_1\underline{\ld}_1\underline{x}^{\underline{\al}_1}+\cdots+\kappa_m\underline{\ld}_m\underline{x}^{\underline{\al}_m}$$
with $\underline{\ld}_i=\ld_1^{{a_1}_i}\cdots \ld_d^{{a_d}_i}$.  By virtue of $f$ being in $R_\chi$, we also know that 
$$gf=\chi(g)f=\chi(g)\kappa_1\underline{x}^{\underline{\al}_1}+\cdots+\chi(g)\kappa_m\underline{x}^{\underline{\al}_m}$$
However, since monomials are linearly independent, this means that for each $i$, $\kappa_i\underline{\ld}_i=\chi(g)\kappa_i$, and so $\underline{\ld}_i=\chi(g)$.  Therefore, for each $i$, $\underline{x}^{\underline{\al_i}}$ is in $R_\chi$ and thus also in $X$.  Hence, $X$ spans $R_\chi$ and is a basis.  

\end{proof}

\begin{prop}\label{prop1}

Suppose $S$ is a power series ring over a field $k$ in $d$ variables and $G$ is a cyclic group of order $n$ acting on $R$ with $\Char k$ not dividing $n$.  If $g$ generates $G$ and has a primitive $n$th root of unity as an eigenvalue, then $S^G$ has only trivial semidualizing modules.

\end{prop}

\begin{proof}

By Lemma \ref{completion}, since $\bar{k}\tns S^G$ is a faithfully flat extension of $S^G$, $C$ is a semidualizing $S^G$-module if and only if $\bar{k}\tns_{S^G} C$ is a semidualizing $\bar{k}\tns S^G$-module.  Thus, if there are no nontrivial semidualizing modules for $\bar{k}\tns S^G$, then there are none for $S^G$.  So, we may assume that $k$ is algebraically closed. 

Since $G$ is cyclic and is generated by $g$, a character in $\hm(G,k^\x)$ is completely determined by the image of $g$.  However, $g$ can only be sent to an $n$th root of unity.  Since $k$ is algebraically closed, and since $\Char k$ does not divide $n$, there are $n$ distinct $n$th roots of unity, which form a cyclic group.  Therefore, $G$ is isomorphic to $\hm(G,k^\x)$.  Since class group of $R^G$ is a subgroup of $\hm(G,k^\x)$, this means the class group must be cyclic.    

By the previous lemma, we may write $R=k[x_1,\dots,x_d]$ where $gx_i=\zeta^{\eta_i}x_i$ with $\zeta\in k$ a primitive $|G|$th root on unity.  The assumption tells us that we may assume that $\eta_1=0$.  Define $\chi: G\to k^\x$ by $g\mapsto \ze^{-1}$.  Since $\ze^{-1}$ is a primitive $|G|$th root of unity, $\chi$ generates $\hm(G,k^\x)$.  So, for some $\ld\in\NN$, $\chi^\ld$ generates the class group $L$.  Assume that $\ld$ is as small as possible.  Note that $gx^\ld_1=(\ze x_1)^\ld=\ze^\ld x^\ld_1$, and so $x^\ld_1\in R_{\chi^{-\ld}}$, the reflexive module corresponding to $\chi^{\ld}$.  Since we have chosen $\ld$ to be as small as possible, $|\chi^\ld|=n/\ld$.  Thus, for each $1\le\nu<|\chi^{-\ld}|=n/\ld$, $\ld\nu$ is strictly less than $n$.  Since the smallest power of $x_1$ that is invariant is $n$, this means that $(x_1^\ld)^\nu$ is indivisible.  Therefore, using the monomial grading, the conditions of Corollary \ref{main} and Theorem \ref{theorem1} are satisfied, and thus $R^G$ has no nontrivial semidualizing modules.  Since $S^G$ is the completion of $R^G$, and completion is faithfully flat, we are done by Lemma \ref{completion}.

\end{proof}

We can recover the non modular case of \cite[Corollary 3.21]{Dao11}.

\begin{corollary}

The there exists no semidualizing modules over nonmodular Veronese subrings.  

\end{corollary}

\begin{proof}

 Let $g$ be an $d\x d$ diagonal matrix whose entries are all $\zeta_n$, a primitive $n$th root of unity.  Then the $n$-Veronese subring in $d$ variables is  $R=k[[x_1,\dots,x_d]]^G$ where $G$ is the group generated by $g$.  Since the order of $G$ is $n$, the result follows from the previous proposition.  

\end{proof}

We now come to our main theorem.

\begin{theorem}\label{thm}

If $S$ is a power series  ring over a field $k$ in finitely many variables and $G$ is a cyclic group of order $p^l$ acting on $S$ with $\Char k\ne p$, then $S^G$ has only trivial semidualizing modules.

\end{theorem}

\begin{proof}

By Lemma \ref{lemma}, we may write $R=k[x_1,\dots,x_d]$ where $gx_i=\zeta^{\eta_i}x_i$ with $\zeta\in k$ a primitive $|G|$th root on unity.  We may assume that $\zeta^{\eta_1}$ has the greatest order of all the $\zeta^{\eta_i}$ and set $z=|\ze^{\eta_1}|$.  Since $|\zeta^{\eta_i}|$ is a power of $p$ less than $z$,  we have $|\zeta^{\eta_i}|$ divides $z$ for each $i$, and so $(\ze^{\eta_i})^z=1$.  Thus, viewing $g$ as a diagonal matrix with entries $\ze^{\eta_i}$, $g^z$ is  the identity, and so $n\le z$.  But, $z$ has to be less than $n$, giving us equality.  Hence, $\zeta^{\eta_1}$ is a primitive $n$th root of unity.  However, since our choice of $\zeta$ is arbitrary, we may assume that $\eta_1=1$.  In short, we have $gx_1=\zeta x_1$.  The result follows from the previous proposition.

\end{proof}

The proofs of Theorem \ref{thm} and Proposition \ref{prop1} show that Theorem \ref{theorem1} applies to the class of rings under consideration.  Thus we actually have the following result, which resolves in the affirmative a special case of Conjecture 1.3 in \cite{Gotoetal2013}.

\begin{corollary}

Assume the set up of the previous theorem, and let $D$ be a dualizing module for $S$. If $M$ is a reflexive module of rank 1 and  $M\tns_S \hm_S(M,D)$ is torsionfree, then $M$ is isomorphic to either $S$ or $D$.  

\end{corollary}

\begin{proof}

Since $M$ and $\hm_S(M,D)$ are both elements of the class group, and since Theorem \ref{theorem1} applies, either $M$ or $\hm(M,D)$ is isomorphic to $S$.  In the latter case implies that $M\cong D$.  

\end{proof}

\section*{Acknowledgements}

The author thanks Sean Sather-Wagstaff and Hailong Dao for their useful insights.  The author would also like to thank the referee for their very helpful comments.

\bibliographystyle{amsplain}
\bibliography{SemidualizingPap}

\end{document}